\documentclass[11pt,a4paper]{article}

\usepackage{xcolor}
\usepackage{a4wide}
\usepackage[utf8]{inputenc}
\usepackage[T1]{fontenc}
\usepackage[fleqn]{amsmath}
\usepackage{amssymb}
\usepackage{amsthm}
\usepackage{mathabx}
\usepackage{framed}
\usepackage{verbatim}
\usepackage[shortlabels]{enumitem}
\numberwithin{equation}{section}

\usepackage{hyperref} 

\setenumerate{itemsep=0.005ex}

\theoremstyle{plain}
\newtheorem{theorem}{Theorem}
\newtheorem{lemma}[theorem]{Lemma}

\newtheorem{externaltheorem}[theorem]{Theorem}
\theoremstyle{definition}
\newtheorem*{definition}{Definition}
\newtheorem*{notation}{Notation}
\theoremstyle{remark}
\newtheorem{remark}[theorem]{Remark}

\renewcommand{\L}{{\mathcal{L}}}
\renewcommand{\mod}{\,\mbox{\scriptsize\rm mod}\,}
\renewcommand{\sigma}{z}

\newcommand{\constant}{\mbox{\it\~c}}

\newcommand{\measure}[1]{\lambda\!\left( #1 \right)}
 
\newcommand{\ceil}[1]{\lceil #1 \rceil }

\newcommand{\digits}[2]{\mathcal{L}(#2,#1)} 

\renewcommand{\>}{\rangle}
\newcommand {\base}[2]{\langle{#1};{#2}\rangle}

\newcommand{\card}{\mbox{\raisebox{.20em}{{$\scriptscriptstyle \#$}}}}
\newcommand{\Tmax}{T_{\mbox{max}}}
\newcommand{\reps}{{\cal S}}

\newcommand{\lcm}{\text{lcm}}
\newcommand{\occ}{\text{occ}}
\newcommand{\B}{B}

\begin{document}

\title{On Simply Normal Numbers to Different Bases}
\author{\begin{tabular}{ccc}
Ver\'{o}nica Becher& Yann Bugeaud &Theodore A. Slaman
\\
{\small Universidad de Buenos Aires}&
{\small Universit\'e de Strasbourg}&
{\small University of California Berkeley}
\\
{\small vbecher@dc.uba.ar}&
{\small bugeaud@math.unistra.fr}&
{\small slaman@math.berkeley.edu}
\end{tabular}
}
\date{August 29, 2013}
\maketitle

\begin{abstract}
Let $s$ be an integer greater than or equal to $2$. A real number is simply normal
to base~$s$ if in its base-$s$ expansion every digit $0, 1, \ldots , s-1$ 
occurs with the same frequency~$1/s$.
Let ${\reps}$ be the set of positive integers that are not perfect powers, 
hence $\reps$ is the set $\{2,3, 5,6,7,10,11,\ldots\} $.
Let $M$ be a function from $\reps$ to sets of positive integers such that, 
for each $s$ in $\reps$, if $m$ is in $M(s)$ 
then each divisor of $m$ is in $M(s)$ and if $M(s)$ is infinite then 
it is equal to the set of all positive integers.  These conditions on $M$ 
are necessary for there to be a
real number which is simply normal to exactly the bases $s^m$ such that $s$ 
is in $\reps$ and $m$ is in $M(s)$.
We show these conditions are also sufficient and further establish that the set of real numbers 
that satisfy them has full Hausdorff dimension.  This extends a result
of W.~M.~Schmidt (1961/1962) on normal numbers to different bases.
\end{abstract}

\tableofcontents

\section{Introduction}
\addtocounter{footnote}{1}
In 1909 \'Emile Borel~\cite{Bor09} introduced the notions of simple normality and of normality to an integer base.
Let $s$ be an integer greater than or equal to $2$. 
A real number $x$ whose  expansion in base $s$ is given by
\[
x = \lfloor x \rfloor + \sum_{j \geq 1} a_j\,s^{-j},
\]
where $a_j \in \{0, 1, \ldots , s-1\}$ for $j \ge 1$,  is said to be {\em simply normal} to base $s$ if every digit $d\in\{0, 1, \ldots , s-1\}$ occurs in the sequence $(a_j)_{ j \ge 1}$ with the same frequency $1/s$.  That is, for every such digit $d$,
\[
\lim_{n\to\infty}\frac{\card\{ j:  1\leq j\leq n \text{ and } a_j=d\}}{n}= \frac{1}{s}.
\]
A number $x$ is said to be {\em normal} to base $s$ if it is simply normal to base $s^k$ for every integer~$k\geq 1$.\footnote{To be accurate, the latter definition is not  
the one originally given by Borel, but equivalent to it.}


 Borel established that almost all real numbers, 
with respect to the Lebesgue measure,
are normal to every integer base greater than or equal to $2$.  
Several equivalent definitions of normality
are given in the monograph~\cite{Bug193}. 

Are there numbers that are {\em simply normal} to  arbitrarily different  bases?
This question was implicit in the literature and hitherto only partially answered.
Recall that two positive integers  are multiplicatively 
dependent when one is a rational power of the other.
It is already known that for any given set 
of bases closed under multiplicative dependence there are uncountably 
many numbers that are simply normal to each base in the given set and 
not simply normal to any base in its complement. 
The historical trace of this result goes back first to a 
theorem of  Maxfield~\cite{Max53} showing that
normality to one base implies normality to another 
when the two are multiplicatively dependent.
Then Schmidt~\cite{Sch61}, improving a result by Cassels~\cite{Cas59}
and his previous result~\cite{Sch60},
showed that for any set of bases closed under multiplicative dependence, 
the set of real  numbers that are normal to every base in the given set  
but not normal to any base in its complement 
is uncountable.  Lastly, Becher and Slaman~\cite{BecSla13}  
established the analogous theorem denying simple normality instead of normality.
These results, however, do not settle the behavior of simple 
normality to bases within multiplicative-dependence equivalence classes.

Already in 1957, Long \cite{Lon57} proved that if a real number is simply normal to base $s^m$ for
infinitely many exponents~$m$, then it is normal to base $s$, 
hence simply normal to base $s^m$ for
every positive integer $m$.  A straightforward analysis, see \cite[Lemma~4.3]{Bug12a}, shows that
for any base $s$ and exponent $m$, 
simple normality to base $s^m$ implies simple normality to base
$s$.  Hertling \cite{Her02} investigated the converse and concluded 
that simple normality to base $r$
implies simple normality to base $s$ if and only if $r$ is a power of $s$.  This leaves open the
question of whether for any given base $s$ there are numbers 
that are simply normal to bases $s^m$
for just finitely many positive integers $m$.  In the following theorem we settle
the characterization of simple normality to different bases and considerably extend Schmidt's
result \cite{Sch61}.

\begin{theorem}\label{1}
  Let ${\reps}$ be the set of positive integers that are not perfect powers, hence $\reps$ is the
  set $\{2,3, 5,6,7,10,11,\ldots\} $.  Let $M$ be a function from $\reps$ to sets of positive
  integers such that, for each $s$ in $\reps$, if $m$ is in $M(s)$ then each divisor of $m$ is in
  $M(s)$ and if $M(s)$ is infinite then it is equal to the set of all positive integers.  There is a
  real number~$x$ such that, for every integer $s$ in $\reps$ and every positive integer $m$, $x$ is
  simply normal to base $s^m$ if and only if $m$ is in~$M(s)$.  Moreover, the set of real
  numbers~$x$ that satisfy this condition has full Hausdorff dimension.
\end{theorem}

\noindent Observe that when $M(s)$ is empty the real number $x$ is not simply normal to base~$s$.

The proof of  Theorem~\ref{1} uses both combinatorial and analytic tools within a global construction.
First, consider the restricted problem, for each base $s\in\reps$,  
of ensuring simple normality of the real number $x$ to
each of the finitely many numbers $s^m$ for $m\in M(s)$, ensuring simple normality 
for each of finitely many numbers $r$ which
are multiplicatively independent to $s$, 
and ensuring a failure of simple normality for~$s^n$, where
$n$ does not divide any element of~$M(s)$.  
We construct an appropriate Cantor set such that almost
every element with respect to its uniform measure is simply normal 
to each base~$s^m$, for $m\in M(s)$, and is
not simply normal to base~$s^n$.  Then, we use Fourier analysis to prove that almost all elements in this
set are simply normal to every base which is multiplicatively independent to $s$.  The latter technique
was used first by Cassels~\cite{Cas59} to prove that almost all elements of the middle-third Cantor
set (with respect to the Cantor measure) are normal to every base which is not a power of $3$.  It
was independently used by Schmidt~\cite{Sch60} to address every pair of multiplicatively
independent integers, and then extended by Schmidt~\cite{Sch61} and Pollington~\cite{Pol81}.

The main novelty in the proof of Theorem~\ref{1} is the determination of this appropriate Cantor set.  
When its elements are
viewed in base~$s^m$, for $m\in M(s)$, each digit should occur with expected frequency~$1/s^m$, 
and when
viewed in base~$s^n$ there should be a bias for some digits over others.  As in
Bugeaud's~\cite[Theorem~6.1]{Bug12a} proof of Hertling's theorem, 
we work with base~$s^\ell$, where $\ell$ 
is a large common multiple of $n$ and the elements of $M(s)$.  
Among the numbers less than~$s^\ell$,
we find one or two, depending on the parity of $s$, which are balanced when written in any of the
bases $s^m$ 
(that is, all the digits in base $s^m$ appear with equal frequency), 
and which are unbalanced when written in base $s^n$.  We obtain the appropriate Cantor set by working in base~$s^\ell$ and omitting these one or two digits.  It takes a rather interesting combinatorial argument
in modular arithmetic to show that such numbers less than~$s^\ell$ exist.

Given this solution to the restricted problem, we construct a nested sequence 
of intervals by recursion with a unique real number $x$ in their intersection. 
A step in the recursion  involves staying in one of the above Cantor
sets long enough so that  
a large initial segment of the  expansion of $x$ to base~$s^n$  has 
that Cantor set's bias while also ensuring 
that the frequency of each digit in the expansion of $x$ to any of the other bases~$r$
being considered continues its convergence to $1/r$, thus giving simple normality to base~$r$.  
Every base to which $x$ is
required to be simply normal is under consideration from some point 
on in the construction and 
every base to which $x$ is required not to be simply normal is acted upon infinitely often.
In case the function $M$ is computable (which means that for each $s \in {\cal S}$, $M(s)$ can be constructed by finitary means),  then so is~$x$.

Regarding metric results, Pollington~\cite{Pol81} established that, for any given set of bases
closed under multiplicative dependence, the set of real numbers 
that are normal to every base in the
given set but not normal to any base in its complement has full Hausdorff dimension.  
More recently,
it is proved in~\cite[Theorem 6.1]{Bug12a} that, for every integer $s$ greater than or equal to $2$
and every coprime integers $m$ and $n$ with $n$ greater than or equal to $2$, the set of real
numbers which are simply normal to base $s^m$ but not 
simply normal to base $s^n$ has full Hausdorff
dimension.  In fact, the proof applies in the more general case in which $n$ does not divide $m$.
The last statement of Theorem~\ref{1}
ensuring full Hausdorff dimension
considerably extends both results.

\begin{notation}
We  denote by $\reps$  the set of positive integers that are not perfect powers,
so $\reps=\{2,3,5,6,7,10,11,\ldots\} $.
A base is an integer greater than or equal to $2$. 
For a base $s$, let $\B_s =\{0,1,\ldots, s-1\}$ denote the set of digits used to
represent real numbers in base $s$.
For a finite set $V$ of non-negative integers, we denote by  $\digits{V}{\ell}$ the  sequences
$(v_0,\dots,v_{\ell-1})$ of $\ell$ many  elements of $V$.  
We refer to such sequences as blocks
and  denote the length of a  block $w$ by $|w|$.  
For $w\in\digits{V}{\ell}$, we denote by $(w;m)$ the  sequence of blocks of length $m$ 
  whose concatenation is the largest prefix of $w$ whose length is a
  multiple of $m$.
We use repeatedly the observation that, for a base~$s$ and positive integers
$\ell$ and $n$ such that $\ell$ is a multiple of~$n$,  a block of length~$\ell$
on $\B_s$ can be seen as a block of length $\ell/n$ on $\B_{s^n}$. 
Furthermore, we sometimes identify the block $b_0 \ldots b_{\ell - 1}$ on $\B_s$
with the integer $b_0 s^{\ell - 1} + \ldots + b_{\ell - 2} s + b_{\ell - 1}$. 
We use the convention that a set is finite if it is empty or it has finitely many elements. 
\end{notation}

\section{Lemmas}
\label{sec:Lemmas}

We start with a collection of lemmas which deal with one single base $s$ and its powers.
We may think of $s$ as an element of $\reps$ but the lemmas  apply to any integer base.

\subsection{Residue equivalence}
\label{sec:residues}

\begin{definition}
  Let $X$, $Y$ be sets of non-negative integers and  let $M$ be a set of positive  integers.  
  Then $X$ and $Y$  are \emph{residue equivalent} 
   for  $M$ if and only if, for every $m$ in $M$ and every integer $r$ with $0\leq r<m$, the sets
  $\{x:x\in X\text{ and } x\equiv r\text{ mod } {m}\}$ and 
  $\{y:y\in Y\text{ and } y\equiv r\text{ mod } {m}\}$ have
  the same cardinality.  When $M=\{m\}$ is a singleton, we  say that $X$ and $Y$ are residue
  equivalent for $m$.
\end{definition}

Instead of directly considering $M$ as a set of positive integers, 
we first consider $M$ as a collection of residue
classes modulo $n$, with multiplicity.

\begin{definition}
  A multiset $M$ of residues mod $n$ is \emph{fair} if there is a positive integer $k$ such that
  $M$ is the multiset in which each integer between $1$ and $n-1$ appears with multiplicity $k$.
\end{definition}

Observe that in case $n$ is $1$, the only fair multiset $M$ of residues mod $1$ is  the empty set.
For a fair multiset $M$, we consider the collection of sums of elements of $M$.

\begin{definition}
Let $n, \sigma, v, k$ be positive integers.  We denote by  $p(n, \sigma, v, k)$
the number of ways that $\sigma$ can be written as a sum of $v$ elements
from the multiset $\{1, \ldots ,1,  \ldots, n-1, \ldots, n-1\}$ in which every 
integer between $1$ and $n-1$ is repeated exactly $k$ times.
\end{definition}

For example, a rapid check shows that $p(3,6,4,2) = 1$ and $p(2, 1, 1, k) =k$ for $k \ge 1$. 
Let $\phi$ denote Euler's totient function: $\phi(n)$ counts the number of positive integers  
less than or equal to $n$ that are relatively prime to $n$.

The following combinatorial theorem, kindly communicated to us
by Mark Haiman, is the key tool for the proof
of Lemma~\ref{3} below.

\begin{externaltheorem}[Haiman~\cite{Hai13}]\label{2}
For any $n$ and $k$ positive integers, we have 
\[
\sum_{\substack{\text{$\sigma$ : $n$ divides $\sigma$}\\ \text{$v$ : $v$ even}}} 
p(n, \sigma, v, k)\quad -
\sum_{\substack{\text{$\sigma$ : $n$ divides $\sigma$}\\ \text{$v$ : $v$ odd}}}
p(n, \sigma, v, k) = n^{k-1}\phi(n).
\]
\end{externaltheorem}

\begin{proof}
The generating function for $p(n, \sigma,v, k)$ is given by
\[
\sum_{s \ge 1, v \ge 1} p(n, \sigma,v, k) x^v q^s=\prod_{j=1}^{n-1}(1+xq^j)^k.
\] 
To calculate
\[d(n,k) = \sum_{\substack{\text{$\sigma$ : $n$ divides $\sigma$}\\ \text{$v$ : $v$ even}}} 
p(n, \sigma,v, k)\quad -
\sum_{\substack{\text{$\sigma$ : $n$ divides $\sigma$}\\ \text{$v$ : $v$ odd}}}
p(n, \sigma,v, k), 
\]
 set $x = -1$ and choose $q$ to be an $n$-th root of unity.  Then,
 averaging over all the $n$-th roots of unity, we obtain 
\[
 d(n,k) = \frac{1}{n} \,  \sum_{w^n = 1} \prod_{j=1}^{n-1} (1 - w^j)^k.
\]
If $w$ is not a primitive $n$-th root of unity, then $w^j=1$ for
some positive integer $j$ less than $n$ and the 
above product vanishes.  
If $w$ is a primitive $n$-th root of unity, then the above product is 
equal to the $k$-th power of
\[
 \prod_{j=1}^{n-1} (1 - e^{2 \pi i j/n}).
\]
Setting $z=1$ in the equality 
\[
\prod_{j=1}^{n-1} (z - e^{2 \pi i j/n}) = (z^n - 1)/(z - 1) = 1 + z + z^2 + \ldots + z^{n-1}, 
\]
we get that
\[
\prod_{j=1}^{n-1} (1 - e^{2 \pi i j/n}) = n.
\]
It then follows that $d(n,k) = n^{k-1} \phi(n)$.
\end{proof}

The next lemma extends the following easy observation.
Let $m_1, m_2$ and $n$ be positive integers such that $n$ does not
divide $m_1$ nor $m_2$. Then, the sets $\{0, m_1 + m_2\}$
and $\{m_1, m_2\}$ are residue equivalent for $m_1$
and for $m_2$, but not for $n$. 

\begin{lemma}\label{3}
Let  $M$ be a non-empty finite set of positive integers and  $n$ be a positive integer that does
not divide any element of $M$. Then, there are sets $X$ and $Y$ of non-negative integers 
which are residue equivalent for $M$ and not residue equivalent for $n$.
\end{lemma}

\begin{proof}
Without loss of generality, we may assume that the multiset of residues of the elements of  $M$
modulo $n$ is fair.  If necessary, $M$ can be extended to a set with this property and proving the
lemma for this larger set also verifies it for $M$.

Let $E(M)$ be the multiset of non-negative integers   that can be expressed as sums of evenly many
elements of $M$, where the multiplicity of each element is the number of ways that it can be
expressed as such a sum. Here, we adopt the convention that the empty sum is even and has value
$0$.  Likewise, let $O(M)$ be the analogous multiset defined using sums of oddly
many elements of $M$.

Write $M=\{m_1, m_2, \ldots , m_k\}$ and $M_j=\{m_i:i\leq j\}$ for $j=1, \ldots , k$, thus
$M=M_k$. Proceed by induction on $j$ to show that $E(M_j)$ and $O(M_j)$ are residue equivalent for
$M_j$.  Observe that $E (\{m_1\}) =\{0\}$ and $O (\{m_1\}) =\{m_1\}$ are residue equivalent for
$\{m_1\}$.  Let $j\leq k-1$ be such that $E(M_j)$ and $O(M_j)$ are residue equivalent for $M_j$.
Then we have
\begin{align*}
  E(M_{j+1})=&E(M_j)\cup\{m_{j+1}+o:o\in O(M_j)\}\\
  O(M_{j+1})=&O(M_j)\cup\{m_{j+1}+e:e\in E(M_j)\}.
\end{align*}  
Observe that $E(M_j)$ and $O(M_j)$ are residue equivalent for $M_j,$ 
and $\{m_{j+1}+o:o\in
O(M_j)\}$ and $\{m_{j+1}+e:e\in E(M_j)\}$ are also residue equivalent for $M_j.$ 
Consequently, $E(M_{j+1})$ and $O(M_{j+1})$ are residue equivalent for $M_j.$

Observe that
both $ E(M_{j+1})$ and $O(M_{j+1})$ are residue equivalent 
to $E(M_j)\cup O(M_j)$ for $m_{j+1}$,
hence residue equivalent to each other for $m_{j+1}$.  This implies
that $E(M_{j+1})$ and $O(M_{j+1})$ are residue equivalent for $M_{j+1}.$
By an immediate induction, we get that $E(M)$ and $O(M)$ are  
residue equivalent for $M$.

Since the multiset  $M$ of residues modulo $n$ is fair, 
we deduce from Theorem~\ref{2}  
that $E(M)$ and $O(M)$ have different counts for the residue $0$ modulo $n$. 
Hence, the multisets $E(M)$ and $O(M)$ are not
residue equivalent for $n$.
Define $X$ and $Y$ as the sets consisting of the minimal non-negative integers such that 
$X$ is residue equivalent to $E(M)$ for $M\cup\{n\}$ and $Y$ 
is residue equivalent to $O(M)$ for $M\cup\{n\}$.
\end{proof}

\subsection{Block equivalence}
\label{sec:block-equivalence}

\begin{notation}
For a finite, non-empty set $M$ of positive integers, $\lcm(M)$ denotes the least common
multiple of the elements of $M$. We set $\lcm(\emptyset) = 1$. 
\end{notation}

\begin{definition}
  Let $s$ be a base and let $M$ be a set of positive integers.  Let $u$ and $v$ be blocks in
  $\digits{\B_s}{\ell}$, where $\ell$ is a multiple of $\lcm(M)$.  Then, $u$ and $v$ are \emph{block
  equivalent} for $M$ if and only if, for each $m\in M$ 
  and for each block $z$ of length $m$ on $\B_s$, the
  number of occurrences of $z$  in $(u;m)$ is the same as the number of occurrences of $z$ in
  $(v;m)$.  When $M=\{m\}$ is a singleton, we say that $u$ and $v$ are block equivalent for $m$.
\end{definition}

\begin{lemma}\label{4}
Let $s$ be a base, 
$M$ be a finite set of positive integers and  
$n$ be a positive integer that does not divide any element of $M$.
There are blocks $u$ and $v$ of digits in base  $s$ 
which are block equivalent for $M$ and not block equivalent for $n$,
and their length is a multiple of each of the elements in $M$ and $n$.
\end{lemma}

\begin{proof}
The first possibility is that $n$ is equal to 1 and hence $M$ is empty.  Then the two blocks $(0)$
and $(1)$ of length $1$ satisfy the conclusions of the lemma.
  
The second possibility is that $n$ is greater than $1$.   As in the proof of Lemma~\ref{3},   enlarging $M$ if necessary, 
 we may assume that the multiset of residues of $M$ modulo $n$ is fair, and hence not empty.
 By Lemma~\ref{3}, let $X$ and $Y$ be sets of non-negative 
 integers that are residue equivalent for $M$ but not for $n$. 
Use concatenations of $\card{X}$ many blocks of length $\ell$, where $\ell$ is a multiple of
 $\lcm(M\cup\{n\})$ and strictly greater than the maximum of $X$ and $Y$.  
We represent an element $x\in X$ by the block $w_x$ consisting 
of $\ell-1$ many $0$s  and a unique $1$ at position $x$.  
Recall that the initial position is numbered by $0$ and the last position by $\ell-1$.  
We concatenate these blocks in the order of the elements of $X$, but any order would~do.  
We define the block for $Y$ similarly.

For each $x$ in $X$ and each $m$ in $M$, the sequence $(w_x;m)$ is 
composed of $\ell/m$ many blocks.  
All but one of these are the identically equal to $0$ block.  The remaining element of $(w_x;m)$ 
contains a $1$ at position $r$, where $r=x\mod{m}$.  
Since $X$ and $Y$ are residue equivalent for
 $M$, the blocks representing $X$ and $Y$ are block equivalent for $M$.
 Similarly, the blocks representing $X$ and $Y$ 
 are not block equivalent for $n$, by the hypothesis
 that $X$ and $Y$ are not residue equivalent for $n$ and the above argument.
\end{proof}

We point out that the two blocks $u$ and $v$  
defined in the proof of the  Lemma~\ref{4} are  binary blocks, 
then a fortiori  blocks of digits in any base~$s$.

\begin{definition}\label{2.2}
  Let $V$ be a finite set and let $w$ be a block in $\digits{V}{\ell}$.  For $v\in V$, let
  $\occ(w,v)$ be the number of occurrences in $w$ of $v$.  The \emph{simple discrepancy} of
  $w$ for the set $V$ is
  \[
  D(w,V)=\max\left\{\left|\frac{\occ(w,v)}{\ell}-\frac{1}{\card{V}}\right|:v\in V\right\}.
  \]
\end{definition}

\begin{definition}\label{2.2}
A block $w$ of length $\ell$  of digits in a finite set $V$  is \emph{balanced} 
for an integer~$m$ and $V$
if $\ell$ is a multiple of~$m$ and $D((w;m),\digits{V}{m})=0$.  
A block~$w$ is {\em balanced} for a set~$M$ of integers and  a finite set~$V$
if it is balanced for every integer in~$M$ and~$V$.
 A set $W$ of blocks of length~$\ell$ is  balanced for a set~$M$ of integers and a finite set~$V$
if $\ell$ is a multiple  of each element of $M$ and the concatenation of the blocks in~$W$ 
(in any order) is balanced for $M$~and~$V$.
\end{definition}

Suppose that $W$ is a subset of $\digits{\B_s}{\ell}$ and $\ell$ is divisible by~$m$.  Consider the
uniform measure $\mu_W$ on the infinite sequences of elements of $W$.  If $W$ is balanced for $m$
and $B_s$, then $\mu_W$-almost-every infinite sequence of elements from $W$ is simply normal to
base~$s^m$, when parsed as an infinite sequence of elements in $\digits{\B_s}{m}$,

\begin{lemma}\label{5}
Let $s$ be a base, $M$ be a finite set of positive integers 
and $n$ be a positive integer that does not divide any element of~$M$. 
There is a set $U=U(s, M, n)$ and a positive integer~$\ell_U$, which may be chosen
arbitrarily large, such that 
\begin{itemize}
\item $U$ is balanced for $M$ and not balanced for $n$,
\item if  $s$ is odd,  then $U=\digits{B_s}{\ell_U}\setminus \{z\}$  with  $z$ even,
\item if   $s$ is even, then  $U=\digits{B_s}{\ell_U}\setminus \{z,\tilde{z}\}$ 
with  $z$ even, $\tilde{z}$ odd and $z<\tilde{z}$.
\end{itemize}
Furthermore,  $\ell_U$ and $U$ are uniformly computable from of $s$, $M$ and $n$.
\end{lemma}

The fact that $U$ and $\ell_U$ may be chosen
arbitrarily large is crucial in Section~\ref{sec:HD}. 

\begin{proof}
If $s$ is equal to $2$ and $n$ is equal to $1$ then let $u$ and $v$ be the $n$-inequivalent blocks $(0,1)$ and $(1,1)$.
Otherwise, let  $u$ and $v$ be blocks of digits in base $s$ ensured by Lemma~\ref{4}  
to be block equivalent for $M$ but not for~$n$. The length $\ell$
of the blocks $u$ and $v$ is a  
positive integer that is divisible by all of the elements of $M$ and also by $n$.
Fix any positive integer $c$ and let $w_0$ be a 
block of length $2 c \ell s^\ell$ obtained by concatenating 
$2c$ many instances of each of the $s^{\ell}$ elements of $\digits{\B_s}{\ell}$ in some order.  
By symmetry, each element of $s$ occurs in $w_0$ exactly as often as any other
element does.  Similarly, if $k$ divides $\ell$, then $(w_0;k)$ can be obtained by concatenating the
elements of $\digits{\B_{s^k}}{\ell/k}$ in the order naturally 
induced by $w_0$ and so each digit in base
$s^k$ appears in $(w_0;k)$ exactly as often as any other digit does.
Thus, for each $k$
that divides $\ell$, the block $w_0$   is balanced for $k$.

Now let $W$ be the set of blocks  in $\digits{\B_s}{2c\ell s^\ell}$
obtained by concatenation of $2c$ many instances of $u$ and $2c$ many instances of each of the elements in 
$\digits{\B_s}{\ell}\setminus\{v\}$.
Observe that each element $w$ in $W$ 
consists of $4c$ instances of $u$ and  $2c$ instances of  each of the blocks in $\digits{\B_s}{\ell}\setminus\{u,v\}$.
No  instances of $v$ have been used.
Since $u$ and $v$ are block equivalent for~$M$, $w_0$ is block equivalent to every $w\in W$ for~$M$. 
Since $w_0$ is balanced for $M$, each $w\in W$ is also balanced for~$M$.  
Similarly, since $u$ and $v$ are not block equivalent for $n$ and $w_0$ 
is balanced for~$n$, each $w\in W$  is not balanced for~$n$. 
In fact, all the elements in $W$ are identically imbalanced for~$n$.
Then,  there is a block $t\in\digits{\B_s}{n}$ and a positive rational constant $\gamma$ such that
 for any two   blocks $w$ and $\tilde{w}$ in $W$,  
\[
\frac{\occ((w;n),t)}{|w|/n}=
\frac{\occ((\tilde{w};n),t)}{|\tilde{w}|/n}<\frac{1}{s^n}-\gamma.
\]

Let $z$ and $\tilde{z}$ be the lexicographically least pair of blocks in $W$ such that $z$ ends with
an even digit, $\tilde{z}$ ends with an odd digit and $z$ is less than $\tilde{z}$.  
The existence of these blocks  $z$ and $\tilde{z}$  follows from the fact that  the blocks 
$u$ and $v$ have  length $\ell$ greater than or equal to $2$. 
This is ensured by the choice of $u$ and $v$ in the special case of $s=2$ and $n=1$, 
and by Lemma~\ref{4} in all the other cases.

If $s$ is even, then $z$ is even and $\tilde{z}$ is odd.
If $s$ is odd, then $z$ is even, since the sum of its digits is even (we have
concatenated an even number of instances of each block).
Let  the length $\ell_U$ be equal to $2c\ell s^\ell$ 
(which is the length of $z$ and $\tilde{z}$).
If  $s$ is odd,  let $U=\digits{B_s}{\ell_U}\setminus \{z\}$. 
If   $s$ is even, let $U=\digits{B_s}{\ell_U}\setminus \{z,\tilde{z}\}$. 

We argue for the case $s$ is even. 
Since $U \cup \{z, \tilde{z}\}$ is balanced for $M$ and both $z$ and $\tilde{z}$ are
also balanced for $M$, we deduce that $U$ is also balanced for $M$.  
Similarly, $U \cup \{z, \tilde{z}\}$ is balanced for $n$ and 
$z$ and $\tilde{z}$ are identically imbalanced for $n$, thus $U$ is not balanced for $n$.
The case $s$ is odd is similar.

Finally, the computability of $U$ follows from the fact that  
$z$ and $\tilde{z}$ are  uniformly computable in terms of $s$, $M$ and $n$.
\end{proof}

\subsection{A lower bound for simple discrepancy in Cantor sets}

Our next lemma is a classical statement saying that, for a finite set $V$, if $\ell$ is large enough
then a large proportion of blocks of length $\ell$ of digits from the set $V$
have small  simple discrepancy.

\begin{lemma}[see Theorem~148, \protect{\cite{HarWri08}}]\label{6}
  For any finite set $V$, for any positive real numbers $\epsilon$ and~$\delta$,
  there is a positive integer~$\ell_0$ such that for all $\ell\geq \ell_0$,
  \[
    {\card}\Bigl\{v\in \digits{V}{\ell}: D(v,V)<\epsilon\Bigr\} >(1-\delta) (\card{V})^\ell.   
  \]
  Furthermore, $\ell_0$ is a computable function of $V$, $\epsilon$ and~$\delta$.
\end{lemma}

Our second lemma will be used to ensure simple normality with respect 
to bases $s^m$, with $m$ in a finite set $M$.

\begin{lemma}\label{7}
  Let $s$ be a base, $M$ be a finite set of positive integers and $n$ be a positive integer that does
  not divide any element of $M$.  Let $U$ be as in Lemma~\ref{5} and let $\ell_U$ be
  the length of the elements of $U$.

  For any positive real numbers~$\epsilon$ and~$\delta$, there is a positive integer 
$\ell_0$ such that for all $\ell\geq \ell_0$,
  \[
  \card\Bigl\{u\in \digits{\B_s}{\ell \ell_U}: (u;\ell_U)\in\digits{U}{\ell}
  \text{ and }\forall m\in M, D\bigl((u;m),\digits{\B_s}{m}\bigr)<\epsilon\Bigr
  \} >(1-\delta) (\card{U})^{\, \ell}.
     \]
 Furthermore, $\ell_0$ is a computable function of $s$, $M$, $n$, $\epsilon$ and~$\delta$.
\end{lemma}

Keeping its notation, Lemma~\ref{7} asserts that, if $\ell$ is large enough, 
then, an arbitrarily large proportion of the $(\card{U})^{\, \ell}$ blocks of length $\ell$
of elements of the set $U$ has, for each $m\in M$, 
the property that, viewed as blocks of 
length $\ell \ell_U / m$ of digits in $\{0, 1, \ldots , s^m - 1\}$, they have arbitrarily
small simple discrepancy for the base $s^m$.
This holds because
both, $z$ and $\tilde{z}$, are balanced for $m$.

\begin{proof}
Consider a block $u\in \digits{\B_s}{\ell  \ell_U}$ such that $(u;\ell_U)\in\digits{U}{\ell}$.
Let  $m\in M$ and $d\in\digits{\B_s}{m}$.  
We count the number of occurrences in $(u;m)$  of $d$ ,
\[
\occ((u;m),d)=\sum_{w\in U}\occ((u;\ell_U),w) \occ((w;m),d).
\]
If $D((u;\ell_U),U)<\epsilon_1$ then, by the definition of discrepancy $D$, for all $w\in U$, 
\[
\frac{\occ((u;\ell_U),w)}{\ell}<\frac{1}{\card{U}}+\epsilon_1. 
\]
Then, 
\[
\occ((u;m),d)<\ell\left(\frac{1}{\card{U}}+\epsilon_1\right) \sum_{w\in U}\occ((w;m),d).
\]
Since $U$ is balanced for $M$,
\[
\sum_{w\in U}\occ((w;m),d)=\frac{\card{U}(\ell_U/m)}{\card{\digits{\B_s}{m}}}
=\frac{\card{U}(\ell_U/m)}{s^m}.
\]
Then,
\begin{align*}
  \occ((u;m),d)<&\ell\left(\frac{1}{\card{U}}+\epsilon_1\right) \frac{\card{U}(\ell_U/m)}{s^m}\\
  <& \ell \left(\frac{1}{s^m}+\frac{\epsilon_1\card{U}}{s^m}\right)\frac{\ell_U}{m}.
\end{align*}
Since $\ell \ell_U/m$ is the length of $(u;m)$, we deduce that 
\[
\frac{\occ((u;m),d)}{|(u;m)|}<\frac{1}{s^m}+\epsilon_1 s^{\ell_U}.
\]
We obtain the analogous lower bound on $\frac{\occ((u;m),d)}{|(u;m)|}$ similarly.  
Thus, for any $\epsilon>0$ and any
$u\in\digits{\B_s}{\ell \ell_U}$ satisfying $D((u;\ell_U),U)<\epsilon s^{-\ell_U}$, we have
$D((u;m),\B_s)<\epsilon$.

Now, let $\epsilon$ and~$\delta$ be positive real numbers.  
By Lemma~\ref{6}, there is an 
$\ell_0$ such that for all $\ell\geq\ell_0$,
\[
  {\card}\Bigl\{u\in \digits{U}{\ell}: D(u,U)<\epsilon s^{-\ell_U} \Bigr\} >(1-\delta) (\card{U})^\ell.   
\]
Equivalently,
\[
  {\card}\Bigl\{u\in \digits{\B_s}{\ell \ell_U}: (u,\ell_U)\in\digits{U}{\ell}\text{ and }D((u;\ell_U),U)
  <\epsilon s^{-\ell_U} \Bigr\} >(1-\delta) (\card{U})^\ell.   
\]
Hence,
\[
  {\card}\Bigl\{u\in \digits{\B_s}{\ell \ell_U}: (u,\ell_U)\in\digits{U}{\ell}\text{ and }
  D((u;m),\digits{\B_s}{m})<\epsilon\Bigr\} >(1-\delta) (\card{U})^\ell,
\]
as required.  
\end{proof}

Our third lemma is the key ingredient to deny simple normality to the base $s^n$.

\begin{lemma}\label{8}
  Let $s$ be a base, $M$ be a finite set of positive integers and $n$ be a positive integer that
  does not divide any element of $M$.  Let $U$ be fixed as in Lemma~\ref{5} and let $\ell_U$ be the
  length of the elements of $U$.

  There is a positive real number~$\epsilon$ and an element $d$ in $\digits{\B_s}{n}$ such that for any positive
  real number~$\delta$, there 
  is a positive integer $\ell_0$ such that for all $\ell\geq \ell_0$,
  \[
  \card\Bigl\{u\in \digits{\B_s}{\ell \ell_U}: (u;\ell_U)\in\digits{U}{\ell}\text{ and }
  \frac{\occ((u;n),d)}{|(u;n)|}< \frac{1}{s^n} -\epsilon \Bigr\}
     >(1-\delta) (\card{U})^{\, \ell},
     \]
 Furthermore, $\epsilon$, $d$ and $\ell_0$ are computable functions of $s$, $M$ and $n$.
\end{lemma}

We point out that $\epsilon$ in Lemma~\ref{8} does not depend on~$\delta$.

\begin{proof}
  We argue as in Lemma~\ref{7}.  Consider a block $u\in \digits{\B_s}{\ell \ell_U}$ and a
  positive real number~$\epsilon_1$ such that $(u;\ell_U)\in\digits{U}{\ell}$ and $D((u;\ell_U),U)<\epsilon_1$.  As
  above, for any $d\in\digits{\B_s}{n}$, we have
\[
\occ((u;n),d)<\ell\left(\frac{1}{\card{U}}+\epsilon_1\right) \sum_{w\in U}\occ((w;n),d).
\]
Since $U$ is not balanced for $n$, there is some $d\in\digits{\B_s}{n}$ 
and a positive constant $c$ such that
\[
\sum_{w\in U}\occ((w;n),d)=\frac{\card{U}(\ell_U/n)}{\card{\digits{\B_s}{n}}}-c
=\frac{\card{U}(\ell_U/n)}{s^n}-c. 
\]
Thus,
\begin{align*}
\occ((u;n),d)<& \ \ell\left(\frac{1}{\card{U}}+\epsilon_1\right)
\left(\frac{\card{U}(\ell_U/n)}{s^n}-c\right)   \text{ and }\\
\frac{\occ((u;n),d)}{ \ell \ell_U/n}<&\ \frac{1}{s^n}-\frac{c}{\card{U}\ell_U/n}+\epsilon_1 s^{\ell_U}.
\end{align*}
If $\epsilon_1$ is sufficiently small, then
\[
\frac{\occ((u;n),d)}{ \ell \ell_U/n} <\frac{1}{s^n}-\frac{c}{2\card{U}\ell_U/n}.
\]
Let $\epsilon$ equal $\frac{c}{2\card{U}\ell_U/n}$.  
The proof is completed by application of Lemma~\ref{6}.
\end{proof}

\subsection{Exponential sums on Cantor sets}

\begin{notation}
  We let $e(x)$ denote ${\rm e}^{2\pi i x}$.  
  We use $\base{b}{r}$ to denote $\ceil{b/\log r}$, where
  $\log$ refers to logarithm in base~${\rm e}$. 
  We say that a rational number $x$ in the unit interval is $s$-adic if
  $x=\sum_{j=1}^ad_js^{-j}$ for digits $d_j$ in $\{0,\dots,s-1\}$.  In this case, we say that $x$
  has precision $a$.
\end{notation}

\begin{lemma}[Hilfssatz~5, \cite{Sch61}]\label{9}
Let $r$ and $s$ be multiplicatively independent bases.  There is a constant $c$, with
$0<c<1/2$, depending only on $r$ and $s,$ such that for all positive integers $k$ and $l$ with
$l\geq s^k$ and for every positive integer $n$,
\[
  \sum_{p=0}^{n-1}\prod_{q=k+1}^\infty |\cos(\pi r^p l/s^q)|\leq 2 n^{1-c}.    
  \]
Furthermore, $c$ is a computable function of $r$ and $s$.\footnote{Actually, Schmidt asserts 
the computability of $c$ in separate paragraph (page 309 in the same article): 
``Wir stellen zun\"achst fest,
da\ss\  man mit etwas mehr M\"uhe Konstanten $a_{20}(r, s)$ aus Hilfssatz~5 explizit berechnen
k\"onnte, und da\ss\ dann $x$ eine eindeutig definierte Zahl ist.''}
\end{lemma}

Lemma~\ref{10}  is our  analytic tool to control discrepancy for multiplicatively independent
bases.  It originates in Schmidt's work~\cite{Sch61}.  Our proof adapts the version  given by Pollington~\cite{Pol81}.

\begin{definition}\label{2.4}
For integers $a, \ell$, sets $R, T$ and a real number $x$, set
\[
A(x,R,T,a,\ell)=\sum_{t\in T}\;\sum_{r\in R}\;
  \Bigl|\sum_{j=\base{a}{r}+1}^{\base{a+\ell}{r}} e(r^j t x)\Bigr|^2.
\]
\end{definition}

\begin{lemma}\label{10}
Let $s$ be a base greater than $2$.
If $s$ is odd, then let $U$ be $\B_s\setminus \{z\}$ for some $z$ in $B_s$ such that  $z$ even.
Else, if $s$ is even, then let $U$ be $\B_s\setminus \{z,\tilde{z}\}$ 
for some $z$ and $\tilde{z}$ in $\B_s$ 
such that $z$ is even and $\tilde{z}$ is odd.
Let $R$ be a finite set of bases multiplicatively independent to $s$, 
$T$ be a finite set of non-zero integers and $a$ be a non-negative integer.  
Let $x$ be $s$-adic with precision~$\base{a}{s}$.  

For every positive real number~$\delta$
there is a length $\ell_0$ such that for all $\ell\geq \ell_0$, 
there are at least $(1-\delta)(\card{U})^{k}$ blocks $v$ in 
$\digits{U}{k}$ for $k=\base{a+\ell}{s}-\base{a}{s}$
such that $A(x_v,R,T,a,\ell)\leq  \ell\,^{2-c(R,s)/4}$, 
for  $x_v=x+s^{-(\base{a}{s}+1)}\sum_{j=0}^{k-1} v_js^{-j}$
and 
$c(R,s)$  the minimum of the constants $c$ in Lemma~\ref{9} for pairs $r,s$ with $r\in R$.  

Furthermore, $\ell_0$ is a computable function of $s$, $U$, $R$ and $T$ and thereby does not depend on $a$ nor on $x$.
\end{lemma}

\begin{proof}
  We abbreviate $A(x,R,T,a,\ell)$ by $A(x)$, abbreviate $(a+\ell)$ by $b$ and
  $\digits{U}{\base{b}{s}-\base{a}{s}}$ by~$\L$.  To provide the needed $\ell_0$ we estimate
  the mean value of $A(x)$ on the set of numbers~$x_v$.  We need an upper bound for
\[
    \sum_{v\in \L}A(x_v)=\sum_{v\in \L}\;\sum_{t\in T}\;\sum_{r\in R}\;
    \Bigl|\sum_{j=\base{a}{r}+1}^{\base{b}{r}} e(r^j t x_v)\,\Bigr|^2
=\sum_{v\in \L}\;\sum_{t\in T}\;\sum_{r\in R}\;
    \sum_{g=\base{a}{r}+1}^{\base{b}{r}}
    \sum_{j=\base{a}{r}+1}^{\base{b}{r}}
    e((r^j-r^g)tx_v).
\]
Our main tool is Lemma~\ref{9}, but it does not apply to all the terms  $A(x_v)$ in the sum.
We will split $\sum_{v\in \L}A(x_v)$  into two smaller sums  $\sum_{v\in \L}B(x_v)$ and $\sum_{v\in \L}C(x_v)$, 
so that  a straightforward analysis applies to the first, and Lemma~\ref{9} applies to the other.
Let $p$ be the least integer satisfying $r^{p-1} \geq 2|t|$ for every $t\in T$ and 
$r^p\geq s^2+1$ for every $r\in R$.

{\everymath={\displaystyle}
\begin{align*}
  \sum_{v\in \L} B(x_v)=\sum_{v\in \L}\sum_{t\in T}\sum_{r\in R}
  \left(
    \begin{array}{ll}
\sum_{g=\base{b}{r}-p+1}^{\base{b}{r}}
\sum_{j=\base{a}{r}+1}^{\base{b}{r}}
e((r^j-r^g)t x_v)&+\\
\sum_{g=\base{a}{r}+1}^{\base{b}{r}}
\sum_{j=\base{b}{r}-p+1}^{\base{b}{r}}
e((r^j-r^g)t x_v)&+\\
\sum_{g=\base{a}{r}+1}^{\base{b}{r}}
\sum_{\substack{j=\base{a}{r}+1\\ |g-j|<p}}^{\base{b}{r}}
e((r^j-r^g)t x_v).
    \end{array}
\right)
\end{align*}
}
We obtain the following bounds.  The first inequality uses that each term in the explicit definition
of $B(x)$ has norm less than or equal to $1$.   Recall, $b=a+\ell$.
\begin{align*}
  \sum_{v\in \L}  |B(x)|&\leq \sum_{v\in \L} \sum_{t\in T} \sum_{r\in R}4p(\base{b}{r}-\base{a}{r}) \\
  &\leq \sum_{v\in \L}  {\card}T\,{\card}R\; 8p\ell \nonumber\\
  &  \leq {\card}T\,{\card}R\;  8p\ell\,(\card{U})^{\base{b}{s}-\base{a}{s}}.
\end{align*}
The other sum is as follows.
\begin{align*}
\sum_{v\in \L} C(x_v) =& \sum_{v\in \L} \sum_{t\in T}\sum_{r\in R}
\sum_{g=\base{a}{r}+1}^{\base{b}{r}-p}\;\;
\sum_{\substack{j=\base{a}{r}+1\\
    |j-g|\geq p}}^{\base{b}{r}-p}
e((r^j-r^g)t x_v)\\
=& \sum_{t\in T}\sum_{r\in R}
\sum_{g=\base{a}{r}+1}^{\base{b}{r}-p}\;\;
\sum_{\substack{j=\base{a}{r}+1\\
    |j-g|\geq p}}^{\base{b}{r}-p}
\sum_{v\in \L}  e((r^j-r^g)t x_v).
\end{align*}
For fixed $j$ and $g$, we have the following identity:
\[
\sum_{v\in \L}\;\;
e((r^j-r^g)tx_v)= e((r^j-r^g)tx) \, 
\prod_{k=\base{a}{s}+1}^{\base{b}{s}}\sum_{d\in U} e\left(\frac{d t (r^j-r^g)}{s^k}\right).
\]
\noindent Since $|\sum_{x \in X} e(x)|=|\sum_{x \in X} e(-x)|$ holds for any 
finite set $X$ of real numbers,
we can bound the sums over $g$ and $j$ as
follows: 
\begin{align*}
  \Bigl|\sum_{v\in \L} C(x_v)\Bigr|\leq&
  \sum_{t\in T}\sum_{r\in R}
\sum_{j=\base{a}{r}+1}^{\base{b}{r}-p}
\sum_{\substack{g=\base{a}{r}+1\\
    |j-g|\geq p}}^{\base{b}{r}-p }\;\;
\prod_{k=\base{a}{s}+1}^{\base{b}{s}}\;
\Bigl|\;
\sum_{d\in U} e\Bigl(\frac{d t(r^j-r^g)}{s^k}\Bigr)
\;\Bigr|\\
\leq&
2
\sum_{t\in T}\sum_{r\in R}
\sum_{j=p}^{\base{b}{r}-\base{a}{r}-p}\;\;
\sum_{g=1}^{\base{b}{r}-\base{a}{r}-p-j}\;\;
\prod_{k=\base{a}{s}+1}^{\base{b}{s}}\;
\Bigl|\;
\sum_{d\in U} e\Bigl(\frac{d tr^{\base{a}{r}}r^g(r^{j}-1)}{s^k}\Bigr)
\;\Bigr|.
\end{align*}
Now, we use the properties of $U$ to show that 
$\Bigl|\sum_{d\in U} e({dx}) \Big|\leq \frac{1}{2}\card{U} |1+e(x)|$. 
 We will also show that $\card{U}$ is even and thus $\frac{1}{2}\card{U}$ is a positive integer.   
 We consider the
odd and even cases for $s$ separately.

Suppose that $s$ is odd.  
By hypothesis $U$ is $\{0,\dots,z-1,z+1,\dots,s-1\}$ and $z$ is even.
Hence, $\card{U}=s-1$ and is even.   We parse our sum in pairs:
\begin{align*}
\sum_{d\in U} e(dx) &=
\sum_{\substack{d\in U\\ d<z}} e(dx)  +
\sum_{\substack{d\in U\\ z<d}} e(dx) \\
&=\sum_{\substack{d\in U\\ d<z\\ \text{$d$ even}}} e(dx)(1+e(x)) +
\sum_{\substack{d\in U\\ z<d\\\text{$d$ odd}}} e(dx)(1+e(x)).
\end{align*}
We conclude that $\Bigl|\sum_{d\in U} e({dx}) \Big|\leq \frac{1}{2}\card{U} |1+e(x)|$.

Suppose that $s$ is even.  Then, $z<\tilde{z}$ and 
$U$ is $\{0,\dots,z-1,z+1,\dots,\tilde{z}-1,\tilde{z}+1,\dots,s-1\}$, 
where $z$ is even and $\tilde{z}$ is odd.  
Hence, $\card{U}=s-2$ and is even.  Again, we parse our sum in pairs:
\begin{align*}
\sum_{d\in U} e(dx) =&
\sum_{\substack{d\in U\\ d<z}} e(dx)  +
\sum_{\substack{d\in U\\ z<d<\tilde{z}}} e(dx) +
\sum_{\substack{d\in U\\ \tilde{z}<d}} e(dx) \\
=&
\sum_{\substack{d\in U\\ d<z\\ \text{$d$ even}}} e(dx)(1+e(x)) +
\sum_{\substack{d\in U\\ z<d<\tilde{z}\\\text{$d$ odd}}} e(dx)(1+e(x)) + 
\sum_{\substack{d\in U\\ \tilde{z}<d\\\text{$d$ even}}} e(dx)(1+e(x)). 
\end{align*}
Again, we conclude that $\Bigl|\sum_{d\in U} e({dx}) \Big|\leq \frac{1}{2}\card{U} |1+e(x)|$.

To simplify the expressions, let $L$ denote $(r^j-1) r^{\base{a}{r}}t$.   Then,
\begin{align*}
  \Bigl|\sum_{v\in \L} C(x_v)\Bigr|
\leq&
2
\sum_{t\in T}\sum_{r\in R}
\sum_{j=p}^{\base{b}{r}-\base{a}{r}-p}\;\;
\sum_{g=1}^{\base{b}{r}-\base{a}{r}-p-j}
\prod_{k=\base{a}{s}+1}^{\base{b}{s}}\;
\Bigl|\;
\sum_{d\in U} e(d L r^gs^{-k})
\;\Bigr|\\
\leq&
2
\sum_{t\in T}\sum_{r\in R}
\sum_{j=p}^{\base{b}{r}-\base{a}{r}-p}\;\;
\sum_{g=1}^{\base{b}{r}-\base{a}{r}-p-j}
\prod_{k=\base{a}{s}+1}^{\base{b}{s}}\;
\frac{\card{U}}{2}\Bigl|\;1+ e( L r^gs^{-k})
\;\Bigr|.
\end{align*}
By the double angle identities,
\[  \Bigl|\sum_{v\in \L} C(x_v)\Bigr|
\leq
2 (\card{U})^{\base{b}{s}-\base{a}{s}}
\sum_{t\in T}\sum_{r\in R}
\sum_{j=p}^{\base{b}{r}-\base{a}{r}-p}\;\;
\sum_{g=1}^{\base{b}{r}-\base{a}{r}-p-j}
\prod_{k=\base{a}{s}+1}^{\base{b}{s}}\;
|\cos(\pi L r^g s^{-k})|. 
\]
The following upper bound on the value of $L$ for  $r$, $j$ and $t$ is ensured by the
choice of $p$.  Let $\Tmax$ be the maximum of the absolute values of the elements of $T$.
\begin{align*}
  Lr^gs^{-\base{b}{s}}&\leq   (r^j -1)r^{\base{a}{r}} t r^g s^{-\base{b}{s}}\\
  &\leq r^j r^{\base{a}{r}} t r^{\base{b}{r} - \base{a}{r} - p - j}  s^{-\base{b}{s}} =
  t r^{\base{b}{r}-p} s^{-\base{b}{s}}\\
  &\leq \Tmax \; r^{\ceil{b/\log r}} \; s^{-\ceil{b/\log s}} r^{-p}\\
  &\leq \Tmax \; r^{1-p}\\
  &\leq 1/2 \qquad\mbox{(an ensured condition on $p$)}.
\end{align*}
Using this upper bound, for every $r$, $j$ and $t$ above, $L r^g s^{-(\base{b}{s}+k)}\leq
2^{-(k+1)}$.  We conclude that 
\[
\prod_{k=\base{b}{s}+1}^\infty |\cos(\pi L r^g s^{-k})|\geq
\prod_{k=1}^\infty |\cos(\pi 2^{-(k+1)})|,
\]
where the right hand side is a positive constant. Then, for all $r$, $j$ and $t$
\begin{align*}
\prod_{k=\base{a}{s}+1}^{\base{b}{s}} |\cos(\pi L r^g s^{-k})|
&=
\prod_{k=\base{a}{s}+1}^\infty |\cos(\pi L r^g s^{-k})|
\;\;\left(\prod_{k=\base{b}{s}+1}^\infty |\cos(\pi L r^g s^{-k})|\right)^{-1}
\end{align*}
which, for an appropriate constant $\constant$, is at most
$\displaystyle{   \constant \;\prod_{k=\base{a}{s}+1}^\infty |\cos(\pi L r^g s^{-k})|}$.
\newline
Now, for $r$, $j$ and $t$, we give a lower bound on the absolute value of $L$.
\begin{align*}
|L|&\geq (r^p-1)r^{\base{a}{r}} = (r^p-1)r^{\ceil{a/\log r}}\\
    &\geq (r^p-1) s^{a/\log s}\\
    &\geq s^{2+a/\log s} \qquad\mbox{(an ensured condition on $p$)}\\
    &\geq s^{\base{a}{s} +1}.
\end{align*}
Using this lower bound, we can apply Lemma~\ref{9}. 
\begin{align*}
\sum_{g=1}^{\base{b}{r}-\base{a}{r}-p-j}
\prod_{k=\base{a}{s}+1}^{\base{b}{s}}\;
\Bigl|\;
\sum_{d\in U} e(d L r^gs^{-k})
\;\Bigr|
\leq&
\sum_{g=1}^{\base{b}{r}-\base{a}{r}-p-j}
\constant \;\prod_{k=\base{a}{s}+1}^\infty |\cos(\pi L r^g s^{-k})|\\
\leq & \qquad 2\constant({\base{b}{r}-\base{a}{r}})^{1-c(R,s)}.
\end{align*}
Then,
\begin{align*}
\Bigl|\sum_{v\in \L} C(x_v)\Bigr|\leq&
2 (\card{U})^{\base{b}{s}-\base{a}{s}} 
\sum_{t\in T}\sum_{r\in R}
\sum_{j=p}^{\base{b}{r}-\base{a}{r}-p}\;\;
2\constant ({\base{b}{r}-\base{a}{r}})^{1-c(R,s)}\\
\leq&
16\constant\; {\card} T\;{\card} R\;\ell^{2-c(R,s)}\; (\card{U})^{\base{b}{s}-\base{a}{s}}. 
\end{align*}
Combining this with the estimate for $|\sum_{v\in\L} B(x_v)|$ and using that $c(R,s)$ is less than $1$, we have
\[
\sum_{v\in\L} A(x_v)\leq (8p+ 16\constant){\card}T {\card}R\;\ell^{2-c(R,s)} (\card{U})^{\base{b}{s}-\base{a}{s}}. 
\]
Therefore, the number of $v\in\L$ such that
$A(x_v)> (8p+ 16\constant){\card}T {\card}R\;\ell^{2-c(R,s)/2}$
is at most equal to 
$\ell^{-c(R,s)/2}\; (\card{U})^{\base{b}{s}-\base{a}{s}}.$
If $\ell$ is greater than $\delta^{-2/c(R,s)}$ then
$\ell^{-c(R,s)/2}<\delta$. In this case, there are at least 
$(1-\delta) (\card{U})^{(\base{b}{s}-\base{a}{s})}$ $v\in\L$ for which
\[
  A(x_v)\leq (8p+ 16\constant){\card}T\, {\card}R \ell^{2-c(R,s)/2}. 
\]
If $\ell$ is also greater than $((8p+ 16\constant)\card{T}\card{R})^{4/c(R,s)}$, then 
there are at least 
$(1-\delta)(\card{U})^{(\base{b}{s}-\base{a}{s})}$  $v\in\L$ for which
\[
A(x_v)\leq \ell\,^{2-c(R,s)/4}.
\]
This proves the lemma for $\ell_0$ equal to the least integer greater than  
$\delta^{-2/c(R,s)}$ and greater than
$((8p+ 16\constant)\card{T}\card{R})^{4/c(R,s)}$.
\end{proof}

\subsection{An upper bound for simple discrepancy  in Cantor sets}

We apply LeVeque's Inequality, which we state for the special case of simple discrepancy of the
digits in the base $s$ expansion of a real number.

\begin{lemma}[\protect{\textbf{LeVeque's inequality}, see Theorem~2.4 on page 111 in~\cite{KuiNie74}}]\label{11}
  Let $s$ be a base, $\ell$~be a positive integer, $w$~be a block in $\digits{B_s}{\ell}$ and $x$~be
  a $s$-adic rational number with precision~$a$. 
Then, letting   $x_w=x+s^{{-a}+1}\sum_{j=0}^{\ell-1} w_js^{-j}$,
  \[
  D(w,\B_s)\leq 
  \Bigl(\frac6{\pi^2}\;\sum_{t=1}^\infty \frac1{t^2} \Bigl|\frac1\ell\; \sum_{j=a+1}^{a+\ell}
  e(t s^jx_w)\Bigr|^2\Bigr)^{\frac13}.
  \]
\end{lemma}

\begin{lemma}\label{12}
Let $s$ be a base, $\epsilon$ be a positive  real number and 
$x$ be a $s$-adic rational number with precision $a$. 
There is a finite set $T$ of positive integers and a positive real number $\gamma$ such that, 
for every positive integer $\ell$ and every block $w$ in $\digits{B_s}{\ell}$, 
letting $x_w=x+s^{{-a}+1}\sum_{j=0}^{\ell-1} w_js^{-j}$,
  \[
\text{ if for all $t\in T$, } \frac{1}{\ell^2}\left|\sum_{j=a+1}^{a+\ell}  e(t s^jx_w)\right|^2<\gamma  
\text{ then } D(w,\B_s)<\epsilon,
   \] 
\end{lemma}
\begin{proof}
Since
\[
 \Bigl|\frac1\ell\; \sum_{j=a+1}^{a+\ell}
  e(t s^jx)\Bigr|^2\leq 1,
  \]
 we get, for each integer $k$,
\begin{align*}
\sum_{t=k+1}^\infty \frac1{t^2} \Bigl|\frac1\ell\; \sum_{j=a+1}^{a+\ell}
e(t s^jx)\Bigr|^2&\leq \sum_{t=k+1}^\infty \frac1{t^2} 
\leq \int_{k+1}^\infty x^{-2} {\rm d}x
\leq \frac1{k+1}.
\end{align*}
Set $k = \ceil{12/(\epsilon^3\pi^2)}$. Assume that
\[
\frac1{\ell^2}\; \Bigl|\sum_{j=a+1}^{a+\ell} e(t s^jx)\Bigr|^2<  \frac{\epsilon^3}{2}
\] 
for all positive integers $t$ less than or equal to $k$. Then,
\begin{align*}
\sum_{t=1}^k \frac1{t^2} \Bigl|\frac1\ell\; \sum_{j=0}^{\ell-1}
e(t s^jx)\Bigr|^2
+\sum_{t=k+1}^\infty \frac1{t^2} \Bigl|\frac1\ell\; \sum_{j=0}^{\ell-1}
e(t s^jx)\Bigr|^2\leq
\sum_{t=1}^k \frac{1}{t^2} \cdot \frac{\epsilon^3}{2} 
+\frac1{k+1} 
\leq \epsilon^3 \frac{\pi^2}{12} +\frac1{k+1}.
\end{align*}
Our choice of $k$ guarantees that
$\bigl(6/\pi^2( ( \epsilon^3 \pi^2 /12) + 1/(k+1))\bigr)^{\frac13}<\epsilon$.
It then follows from Lemma~\ref{11} that $D(w,\B_s)<\epsilon$.
This proves the lemma with $T = \{1, \ldots , k\}$ and $\gamma = \epsilon^3 / 2$. 
\end{proof}

\begin{lemma}\label{13}
Let $s$ be a base greater than $2$.
If $s$ is odd, then let $U$ be 
$\B_s\setminus \{z\}$ 
for some even $z$.
Else, if $s$ is even, then let $U$ be 
$\B_s\setminus \{z,\tilde{z}\}$ 
for some even $z$ and some odd $\tilde{z}$ such that  $z<\tilde{z}$.
Let $R$ be a finite set of bases multiplicatively independent to $s$.
Let $x$ be $s$-adic with precision~$\base{a}{s}$.  

For all positive real numbers $\epsilon$ and~$\delta$ there is a length~$\ell_0$
such that for all $\ell\geq \ell_0$, there are at least
$(1-\delta)(\card{U})^{\base{a+\ell}{s}-\base{a}{s}}$ blocks
$v\in\digits{U}{\base{a+\ell}{s}-\base{a}{s}}$ such that for each $r\in R$,   
the block $u\in\digits{B_r}{\base{a+\ell}{r}-\base{a}{r}}$ 
in the expansion of   
$x+s^{-(\base{a}{s}+1)}\sum_{j=0}^{\base{a+\ell}{s}-\base{a}{s}-1   } v_js^{-j}$ 
 in base~$r$ satisfies $D(u,\B_r)<\epsilon$.

Furthermore, $\ell_0$ is a computable function of $s$,  $U$ and $R$ and thereby  does not depend on
$a$ nor on~$x$.
\end{lemma}

\begin{proof}
Assume  given $s, U, R, x$ and $a$   as in the hypothesis.
Fix $\epsilon$ and~$\delta$  positive real numbers greater than $0$. 
For each  base in $r$ in $R$ consider  Lemma~\ref{12} with input values the base $r$ and the fixed $\epsilon$.  
Fix a finite set of positive integers $T_\epsilon$ and a positive real number $\gamma_\epsilon$ 
that satisfies the conclusion of Lemma~\ref{12} simultaneously for all  bases $r$ in $R$  and  the fixed  $\epsilon$.

Apply Lemma~\ref{10} with input values  $s, U, R, T_\epsilon, x$ and $a$. 
Then there is  an $\ell_0$ such that for all $\ell\geq \ell_0$, there are at least
  $(1-\delta)(\card{U})^{k}$ blocks $v\in\digits{U}{k}$ 
such that  $A(x_v,R,T_\epsilon,a,\ell)\leq \ell\,^{2-c(R,s)/4}$,
where 
$k=\base{a+\ell}{s}-\base{a}{s}$,
$x_v=x+s^{-(\base{a}{s}+1)}\sum_{j=0}^{k-1} v_js^{-j}$ and
$c(R,s)$ is the minimum of the constants $c$ in Lemma~\ref{9} for pairs $r,s$ with $r\in R$.  

Fix $\ell$ be such that $\ell\geq\ell_0$
  and $\ell^{-c(R,s)/4}<\gamma_\epsilon$.  By definition, 
\[
A(x_v,R,T_\epsilon,a,\ell)=
\sum_{t\in  T_\epsilon}\;\sum_{r\in R}\; \Bigl|
\sum_{j=\base{a}{r}+1}^{\base{a+\ell}{r}} e(r^j t  x_v)\Bigr|^2.
\]  
Hence, for each $t\in T_\epsilon$ and for each $r\in R$, 
\[
\frac{1}{\ell^2}\left|\sum_{j=\base{a}{r}}^{\base{a+\ell}{r}} e(r^j t x_v)\right|^2<\gamma_\epsilon.  
\]
Then, by Lemma~\ref{12}, for each $r\in R$, $D(u,\B_r)<\epsilon$,
  where $u$ is the block of digits from position $\base{a}{r}+1$ to position $\base{a+\ell}{r}$ in
  the  expansion of $x_v$ in base $r$.
\end{proof}

\begin{lemma}\label{14}
Let $\epsilon$ be a positive real number,
 $s$ and $r$ be bases and $a$ and $b$ be positive integers such that $a<b$.  
Let $q$ be an $s$-adic rational  number with precision $\base{b}{s}$ and 
$x$ be a real number in the interval $[q,  q+s^{-\base{b}{s}})$.  
Let $u$ and $v$ in $\digits{B_r}{\base{b}{r} - \base{a}{r} +1}$ be, respectively,
the blocks in the expansions of $q$ and $x$ in base $r$ between the positions $\base{a}{r}$ and $\base{b}{r}$.  
Let $p$ be a positive  integer and  let $\tilde{u}$ in $\digits{B_{r^p}}{\base{b}{r^p}-\base{a}{r^p} +1}$
be the block in the expansion of $q$ in base $r^p$ between the  positions $\base{a}{r^p}$ and $\base{b}{r^p}$.

  If $D(u,\B_r)$, $D(\tilde{u},\B_{r^p})$, $2/r^p$ and $3p/|u|$ are all less than~$\epsilon$, then
  $D(v,\B_r)<5\epsilon$.
\end{lemma}

\begin{proof}
Let $\tilde{v}$ in $\digits{B_{r^p}}{\base{b}{r^p}-\base{a}{r^p} +1}$
be the block in the expansion of $x$ in base $r^p$ between the  positions $\base{a}{r^p}$ and $\base{b}{r^p}$.
Since $0\leq x-q<s^{-\base{b}{s}}$, then $0\leq x-q\leq (r^p)^{-\base{b}{r^p}+1}$.  
Any difference between $\tilde{u}$ and $\tilde{v}$ other than in the last two  digits must come from a block of instances of the digit $r^p-1$ in the expansion of $q$ in base $r^p$  at positions preceding its last two.  
 
Since $D(\tilde{u},\B_{r^p})<\epsilon$, at most $(1/r^p+\epsilon)|\tilde{u}|$  
digits in $\tilde{u}$ can be equal to $r^p-1$.  So, $\tilde{u}$ and $\tilde{v}$ agree on all but the last
 $(1/r^p+\epsilon)|\tilde{u}|+2$ digits.  
But then $u$ and $v$ agree on all but the last $(1/r^p +\epsilon )p|\tilde{u}| +3p$ digits. 
 Then, for any $d$ in base $r$, the quantity $|\occ(u,d)-\occ(v,d)|$ is less than
  or equal to $(1/r^p +\epsilon)p|\tilde{u}|+3p$.  Thus,
  \begin{align*}
  \occ(v,d) / |v|\leq& \occ(u,d)/|u|+((1/r^p +\epsilon)p|\tilde{u}|+3p)/|v|\\
  \leq& (1/r+\epsilon)+((1/r^p +\epsilon)p|\tilde{u}|+3p)/|v|\\
  \leq& (1/r+\epsilon)+2(1/r^p +\epsilon)+3p/|v|\\
  \leq& 1/r + 5\epsilon, \text{\ provided $2/r^p$ and $3p/|v|$ are each less than $\epsilon$.}
\end{align*}
The lemma follows.
\end{proof}

In the next two lemmas, we denote by $\measure{I}$ the length 
of a real interval $I$.

\begin{lemma}\label{15}
  For any real interval $I$ and base $s$, there is a $s$-adic subinterval $I_s$ such that
  $\measure{I_s}\geq \measure{I}/(2s)$.
\end{lemma}

\begin{proof}
  Let $m$ be least such that $1/s^m<\measure{I}$.
  Note that $1/{s^m}\geq \measure{I}/ {s}$, since $1/s^{m-1}\geq\measure{I}$.  
  If there is a $s$-adic interval of length $1/ {s^m}$ strictly contained in~$I$, 
  then let $I_s$ be such an interval, and note that $I_s$ has
  length greater than or equal to $\measure{I}/{s}$.  
  Otherwise, there must be a non-negative integer $a$ such that
  $a/s^m$ is in $I$ and neither $(a-1)/s^m$ nor $(a+1)/s^m$ belongs to~$I$.  
  Thus, $2/s^m>\measure{I}$.  
  However, since $1/s^m < \measure{I}$ and $s\geq 2$ then 
  $2/s^{m+1}<\measure{I}$.  
   So, at least one of the two intervals
  $\displaystyle{\left[\frac{sa-1}{s^{m+1}}, \frac{sa}{s^{m+1} }\right)}$ or
  $\displaystyle{\left[\frac{sa}{s^{m+1} }, \frac{sa+1}{s^{m+1} }\right)}$ must be contained in~$I$.
Denote by $I_s$ one with this property. 
  Then, $\measure{I_s}$ is $\displaystyle{\frac{1}{s^{m+1}}=\frac{1}{2s}\frac{2}{s^m}\ >\  \measure{I}/(2s).}$ 
  In either case, the length of~$I_s$ is greater than or equal to~$\measure{I}/(2s)$.
\end{proof}

\begin{lemma}\label{16}
  Let $s$ and $t$ be bases and let $I$ be an $s$-adic interval of length
 $s^{-\<b;s\>}$.  For $a=b+\ceil{\log s + 3\log t}$, there is an $t$-adic subinterval of
 $I$ of length $t^{-\<a;t\>}$.  
 \end{lemma}

\begin{proof}
By the proof of Lemma~\ref{15}, there is an $t$-adic subinterval of $I$ of length
$t^{-(\ceil{-\log_{t}(\mu(I))}+1)}$:
\begin{align*}
  \ceil{-\log_{t}(\measure{I})}+1&= \ceil{-\log_{t}(s^{-\base{b}{s}})}+1\\
  &=\ceil{
    {\base{b}{s}\log s }/{\log t}
      }+1\\
  &\leq \ceil{ b/\log t+ \log s/\log t}+1\\
  &\leq \base{b}{t} +\ceil{\log s/\log t} +1.  
\end{align*}
Thus, there is an $t$-adic subinterval of $I$ of length 
$t^{-(\base{b}{t} +\ceil{\log s/\log t} +1)}$.  Consider $a=b+\ceil{\log s + 3\log t}$,
\begin{align*}
  \base{a}{t} &= \ceil{a/\log t} \\
  &= \ceil{{b+\ceil{\log s + 3\log t}}/{\log t}}\\
  &\geq b/\log t +(\log s+3\log t)/\log t\\
  &\geq \base{b}{t} +\ceil{\log s/\log t} +1.
\end{align*}
This inequality is sufficient to prove the lemma.
\end{proof}

The next remark follows from direct substitution in Lemma \ref{16} above.

\begin{remark}\label{17}
   Let $r$, $s$ and $t$ be bases.  Let $b$ be a positive integer and let
   $a=b+\ceil{\log s+ 3\log t}$.  Then,
\[
   \base{a}{r}-\base{b}{r}\leq \ceil{\log s + 3\log t}/\log r+1
   \leq 2\ceil{\log s + 3\log t}.
\]
\end{remark}

\subsection{Simple discrepancy  and concatenation}

We record the next three observations without proof.

\begin{lemma}\label{18}
  Let $\epsilon$ be a positive real, $r$ be a base, $\ell$ a positive integer and 
$w\in\digits{B_r}{\ell}$ such that   $D(w,\B_r)<\epsilon$.  
Then, for any positive integer $k$ with $k<\epsilon \ell$ 
  and any $u\in\digits{B_r}{k}$, we have $D(wu,\B_r)<2\epsilon$.
\end{lemma}

\begin{lemma}\label{19}
Let $\epsilon$ be a positive real, $r$ be a base, 
and  $(w_j)_{j \ge 0}$ be an infinite sequence of elements from $\{0, 1, \ldots , r-1\}$.
Let $(b_t)_{t \ge 0}$ be an increasing sequence of positive integers.  
Suppose that there is an integer $t_0$ such that,
for all $t>t_0$, we have
$b_{t+1}-b_t\leq \epsilon b_t$ and $D((w_{b_t+1}, \ldots , w_{b_{t+1}}),\B_r)<\epsilon$.
Then, 
  \[
  \lim_{k\to\infty}D((w_0, \ldots , w_k), \B_r) \leq 2\epsilon.
  \]
\end{lemma}

\begin{lemma}\label{20}
Let $\epsilon$ be a positive real, 
$r$ be a base, 
$d$ be a digit in base $r$ and 
$(w_j)_{j \ge 0}$ be an infinite sequence of elements from $\{0, 1, \ldots , r-1\}$.
Let $(b_t)_{t \ge 0}$ be an increasing sequence of positive integers. 
Suppose that there is an integer $t_0$ such that, for all $t > t_0$, we have
 \[
 \frac{\occ((w_{b_t+1}, \ldots , w_{b_{t+1}}),d)}{b_{t+1}-b_t}< \frac{1}{r}-\epsilon.
 \]  
  Then,
  \[
  \liminf_{t\to\infty}\frac{\occ((w_0, \ldots , w_{b_t}),d)}{b_t+1}< \frac{1}{r} - \frac{\epsilon}{2}.
  \]  
\end{lemma}

\section{Existence}
\label{sec:Existence}

  Let   $s\mapsto M(s)$ be given as in the hypothesis of the theorem.  
  If for every $s\in\reps$, the set $M(s)$
  is infinite, then any absolutely normal number satisfies the conclusion of the theorem.  
  Thus, we assume
  that there is at least one $s\in\reps$ for which $M(s)$ is finite.  We construct a sequence
  of intervals by recursion so that the unique real number $x$ in their intersection has the
  properties stated in the conclusion of the theorem.  We define the following 
  sequences indexed by $j$:

\begin{itemize}
\item Fix sequences $((s_j,n_j))_{j \ge 0}$ and $(r_j)_{j \ge 0}$ as follows.  In the sequence
  $((s_j,n_j))_{j \ge 0}$, the integer $s_j$ is an element of $\reps$ such that $M(s_j)$ is finite,
  $n_j$ is a positive integer that does not belong to $M(s_j)$ and every such pair appears
  infinitely often.  The sequence $(r_j)_{j \ge 0}$ is the enumeration of all of the numbers $s^m$, for
  $s\in\reps$ and $m\in M(s)$, in increasing order, including those for which $M(s)$ is infinite.

\item For $j \ge 0$, set $s^*_j = s_j^{\ell_U}$, 
where $\ell_U$ is as in Lemma~\ref{5} for $s_j$,  $M(s_j)$ and $n_j$.

\item For $j \ge 0$ and the pair $(s_j,n_j)$, 
let $d_j$ and $\epsilon_j$ be as guaranteed by the conclusion of
Lemma~\ref{8} for $s_j$, $M(s_j) $ and $n_j$.

\item For $j \ge 0$, let $p_j$ be the least positive integer such that
for each $k$ less than or equal to~$j$, we have $r_k^{p_j}\geq 2(j+1)$.
\end{itemize}

The recursion uses two additional functions denoted by $\ell(j)$ and~$x(j,a,y)$. 
 Let $\ell(j)$ be the least positive integer such that the
  following hold: 

\begin{itemize}
\item For all positive integers $a$ and all $k$ less than or equal to~$j$,
$\base{a+\ell(j)}{r_k}-\base{a}{r_k}$ is greater than $2\ceil{\log s^*_{j-1}+3\log s^*_{j}}(j+2)$.

\item For all positive integers $a$ and all $k$ less than or equal to~$j$,
$\base{a+\ell(j)}{r_k}-\base{a}{r_k}$ is greater than $3p_k(j+2)$.

\item For all positive integers $a$, the conclusion of Lemma~\ref{7} 
with input values $s_j$, $M(s_{j})$ and~$n_{j}$
applies to $\base{a+\ell(j)}{s_j}-\base{a}{s_j}$ for 
$\epsilon=1/(j+1)$ and $\delta=1/4$.

 \item For all positive integers $a$, the conclusion of 
Lemma~\ref{8} with input values $s_{j}$, $M(s_{j})$ and~$n_{j}$ 
 applies to    $\base{a+\ell(j)}{s_j}-\base{a}{s_j}-n_j$, for $\delta=1/4$.

\item The conclusion of Lemma~\ref{13} 
with input values $s^*_j$, 
$U(s_{j}, M(s_{j}), n_{j})$ as in Lemma~\ref{5},
 and  the set   of bases  
$(\{r_k:k\leq j\}\cup \{r^{p_j}_k:k\leq j\})\setminus\{s_j^m:m\in M(s_j)\}$,
applies to $\ell(j)$  for $\epsilon=1/(j+1)$ and~$\delta=1/4$.
\end{itemize}

\noindent For $y$ an $s^*_j$-adic rational number of precision 
$(s^*_j)^{\base{a}{s^*_j}}$, let $x(j,a,y)$ be the least number such that there is a block $w^*$ of length
$\base{a+\ell(j)}{s^*_j}-\base{a}{s^*_j}$ with elements in $U(s_{j},M(s_{j}),n_{j})$ (as as in Lemma~\ref{5})  for which the following hold.  
Let $w$ be the sequence in base~$s_j$ such that $(w;\ell_U)=w^*$:

\begin{itemize}
\item  $\displaystyle{x(j,a,y)=y+(s^*_j)^{-(\base{a}{s^*_j}+1)}\sum_{k=0}^{|w^*|-1}  
  w^*_k(s^*_{j})^{-k}.}$

 \item  For all $m$ in $M(s_j)$, $D\bigl((w;m),\digits{s_j}{m}\bigr)< \frac{1}{j+1}$.

\item $\displaystyle{\frac{\occ((w;n_j),d_j)}{|(w;n_j)|}< \frac{1}{s^{n_j}} -\epsilon_j}$.

\item For all $r$ in $(\{r_k:k\leq j\}\cup \{r^{p_j}_k:k\leq j\})\setminus\{s_j^m:m\in M(s_j)\}$
and for $u$ the block of digits from position $\base{a}{r}+1$ to position $\base{a+\ell(j)}{r}$
 in the base-$r$ expansion of $x(j,a,y)$, we have $D(u,\B_r)<1/(j+1)$.
\end{itemize}

 We proceed by recursion on $t$ to define sequences $(j_t)_{t\ge 0}$, 
$(b_t)_{t \ge 0}$ and~$(x_t)_{t \ge 0}$.  For $t \ge 0$, $j_t$~and $b_t$ are positive integers 
and $x_t$ is a  $s^*_{j_t}$-adic rational number of precision $\base{b_t}{s^*_{j_t}}$.  
The real~$x$, which is the  eventual result of our construction, will be an element of
  $[x_t,x_t+(s^*_{j_t})^{-\base{b_t}{s^*_{j_t}}})$.  
\medskip
\medskip

\noindent{\emph{Initial stage. }} Let $j_0=0,$  $x_0=0$ and $b_0=0$.

\medskip \noindent{\emph{Stage $t+1$.}}  Given $j_t$, $b_t$, $x_t$.  
Consider the two conditions.
\begin{enumerate}
\item For all bases $r\in\{r_k:k\leq j_t+1\}\setminus\{s_{j_t}^m:m\in M(s_{j_t})\}$,
  \[
  \base{b_t+\ceil{\log s^*_{j_t} + 3\log
      s^*_{j_t+1}}+\ell(j_t+1)}{r}-\base{b_t}{r}< \frac{\base{b_t}{r}}{j_t+1}.
  \]
\item For the block $w$ composed of the
  first $\base{b_t}{s_{j_t}}$ digits in the base-$s_{j_t}$ expansion of $x_t$,
\[
\frac{\occ((w;n_{j_t}),d_{j_t})}{|(w;n_{j_t})|}< \frac{1}{s^{n_{j_t}}} - \frac{\epsilon_{j_t}}{2}.
  \]
\end{enumerate}
If both conditions hold, let $j_{t+1}=j_t+1$, let $a=b_t+\ceil{\log s^*_{j_t} + 3\log s^*_{j_t+1}}$
and let $y$ be the left endpoint of the leftmost $s^*_{j_{t+1}}$-adic subinterval of
$[x_t,x_{t}+(s^*_{j_t})^{-\base{b_{t}}{s^*_{j_t}}})$ of length
$(s^*_{j_{t+1}})^{-\base{a}{s^*_{j_{t+1}}}}$.  Otherwise, $j_{t+1}=j_t$, $a=b_t$, and $y=x_t$.
Finally define,
\[
x_{t+1}=x(j_{t+1},a,y) \text{ and } b_{t+1}=a+\ell(j_{t+1}).
\]
\medskip

We check that the construction succeeds.   By Lemmas~\ref{7}, \ref{8} and~\ref{13}, 
the integer $\ell(j)$ is well defined. 
Indeed, in the definition of $x(j,a,y)$, each of these lemmas is applied so that
at least $3/4$ of the blocks being considered are suitable. 
Thus, at least $1/4$
of the blocks being considered are suitable in terms of all three of the lemmas.  It follows that
$x(j,a,y)$ is well defined and that the sequence $x_t$ converges to a limit $x$.

We show that $(j_t)_{t \ge 0}$ tends to infinity with $t$. 
Clearly, the function $t \mapsto j_t$ is non-decreasing.
Suppose that $\lim_{t\to\infty}j_t =h<\infty$ and let $t_0$ be such that $j_{t_0}=h$.  
By the first condition in the definition of the function $\ell$, 
we have $b_{t+1}>b_t$ for $t \ge 0$.  Thus, since
the value of $\ell(j_t)$ does not depend on that of $b_t$, 
there is a stage $t_1>t_0$ such that
for all $t>t_1$ and all  $r\in\{r_k:k\leq  j_t\}$,
 the quantity
\[
\base{b_t+\ceil{\log s^*_{h} + 3\log    s^*_{h}}+\ell(h)}{r}-\base{b_t}{r}
\] 
is less than $\frac{\base{b_t}{r}}{j_t+1}$.
Similarly, for all stages $t>t_0$, we have $x_{t+1}=x(h,b_t,x_t)$.  
Then, by Lemma~\ref{20} for $w$ equal to the expansion of $x$ in base
$s_h^{n_h}$ and the sequence of integers $(c_t) _{t \ge t_0}=(\base{b_t}{s^*_h}(\ell_U/n_h)) _{t \ge t_0}$,
\[
\liminf_{t\to\infty}\frac{\occ((w_0, \ldots , w_{c_t}),d_h)}{c_t+1}<\frac{1}{s_h^{n_h}} -\frac{\epsilon_h}{2}.
\]
Then, there must be a $t>t_1$ such that
\[
\frac{\occ((w_0, \ldots , w_{c_t}),d_h)}{c_t+1}<\frac{1}{s_h^{n_h}} - \frac{\epsilon_h}{2}.  
\]
For such a $t$, the criteria for defining $j_{t+1}=j_t+1$ are satisfied,
contradicting the supposition that $\lim_{t\to\infty}j_t =h$.  Thus, $\lim_{t\to\infty}j_t=\infty.$

Suppose that $s\in\reps$ and $n\not\in M(s)$.  Then, there are infinitely many $j$ such
that $(s_j,n_j)=(s,n)$.
Fix $d$ and $\epsilon$ as guaranteed by the
conclusion of Lemma~\ref{8} for $s$, $M(s)$ and $n$.
 By the previous paragraph, there are infinitely many stages $t$ such that
$s=s_{j_t}$, $n=n_{j_t}$, $d=d_{j_t}$, $\epsilon=\epsilon_{j_t}$ and
\[
\frac{\occ((w_0, \ldots ,w_{\base{b_t}{s^*_{j_t}}(\ell_U/n)}),d)}{{\base{b_t}{s^*_{j_t}}(\ell_U/n)}+1}< 
\frac{1}{s^{n}} - \frac{\epsilon}{2},
\] 
where
$(w_0, \ldots ,w_{\base{b_t}{s^*_{j_t}}(\ell_U/n)})$ is the sequence 
of digits in the base-$s^n$ expansion of $x_t$.
By construction, these are also the sequence of digits in the base-$s^n$ expansion of $x$.
Consequently, $x$ is not simply normal to base $s^n$.

Now, suppose that $s\in\reps$ and $m\in M(s)$.  By the definition of the sequence 
$(r_j)_{j \ge 0}$, fix the
integer~$h$ such that $s^m=r_{h}$.  
Since $\lim_{t\to\infty}j_t=\infty$ we can fix $t_0$ such that $j_t\geq h$ for all $t>t_0$.  
We consider the construction during stages $t+1>t_0$.

There are two possibilities for $s^m$ during stage $t+1$, depending on whether $s_{j_{t+1}}=s$
or not.
Suppose first that $s_{j_{t+1}}\neq s$.  Then, $x_{t+1}$ was chosen so that for the block $u$ of
digits from position $\base{a}{s^m}+1$ to position $\base{a+\ell(j_{t+1})}{s^m}=\base{b_{t+1}}{s^m}$ in
the base-$s^m$ expansion of $x_{t+1}$, we have 
$D(u,\B_{s^m})<1/(j_{t+1}+1)$, where $a$ is $b_t$ or
$b_t+\ceil{\log s^*_{j_t} + 3\log s^*_{j_t+1}}$.  
In the latter case, by the first condition in the
definition of $\ell(j_{t+1})$, 
we deduce that
$\base{a+\ell(j_{t+1})}{s^m}-\base{a}{s^m}$ is greater than $2\ceil{\log
s^*_{j_t}+3\log s^*_{j_{t+1}}} (j_{t+1}+2)$.  It then follows from Remark~\ref{17}
that 
\[
\base{a}{s^m}-\base{b_t}{s^m}<2\ceil{\log
s^*_{j_t}+3\log s^*_{j_{t+1}}}.  
\]
By Lemma~\ref{18}, for the block $v$ of digits in the base-$s^m$
expansion of $x_{t+1}$ between positions $\base{b_t}{s^m}+1$ and $\base{b_{t+1}}{s^m}$,
we have $D(v,\B_{s^m})<2/(j_{t+1}+1)$.  
By construction, we treat the base 
$(s^m)^{p_{j_{t+1}}}$ 
similarly during stage $t+1$
and the same conclusion applies.  

Alternatively, suppose that $s=s_{j_t}$.  Again, for the block $v$ of digits in the base-$s^m$
expansion of $x_{t+1}$ between positions $\base{b_t}{s^m}+1$ and $\base{b_{t+1}}{s^m}$, we have
$D(v,\B_{s^m})<2/(j_{t+1}+1)$ by virtue of the second condition in the definition of $x(j,a,y)$ and
the above observations.  Similarly, this conclusion holds for base ${(s^m)}^{p_{j_t}}.$   (Note, to
keep the discussion simple, we have chosen to ignore the possibility of a difference between
$\base{b_{t+1}}{s^m}$ and $\base{b_{t+1}}{s^\ell_U}(\ell_U/m)$, where $U$ is as is defined during
stage $t+1$.)

Now, we consider the expansion of $x$ in base $s^m$.  For each $t>t_0$, by the definition of the
function $\ell$, Lemma~\ref{14} applies to the digits in this expansion from position
$\base{b_t}{s^m}+1$ to $\base{b_{t+1}}{s^m}$.  Thus, for each of these blocks in the expansion of
$x$ in base $s^m$, the simple discrepancy is less than $10/(j_{t+1}+1)$.  Since $j_t$ tends to
infinity as $t$ increases, 
by Lemma~\ref{19}, $x$ is simply normal to base $s^m$.

\section{Hausdorff dimension}
\label{sec:HD}

Like in \cite{Pol81}, the key tool for estimating the 
Hausdorff dimension of the set defined in Theorem~\ref{1}
is the following result of Eggleston~\cite[Theorem 5]{Egg52}.
 
\begin{lemma}\label{21}
Suppose that, for $k \ge 1$, the set $K_k$ is a linear set
consisting of $N_k$ closed intervals each of length $\delta_k$
and such that each interval of $K_k$ contains $N_{k+1} / N_k$ 
disjoint intervals of~$K_{k+1}$.
If $v_0 \in (0, 1]$ is such that for every $v < v_0$ the sum
\[
\sum_{k \ge 2} \, \frac{\delta_{k-1}}{\delta_k} \, (N_k (\delta_k)^v)^{-1}
\]
converges, then the Hausdorff dimension
of the set $\bigcap_{k \ge 1} K_k$ is greater than or equal to $v_0$. 
\end{lemma}

We analyze the construction of Section~\ref{sec:Existence}.
We keep the notation from that section.

We introduce a positive real number $\eta$ such that
\[
\frac{\log(s_j^* - 2)}{\log s_j^*} \ge \eta
\]
for $j \ge 0$. In view of Lemma~\ref{5} the bases $s_j^*$ may be taken
arbitrarily large, thus $\eta$ can be taken arbitrarily close to $1$.
For convenience, we assume that the sequence $(s_j^*)_{j \ge 0}$ 
is non-decreasing and that $s_0^* \ge 4$. 

Let $t \ge 2$ be an integer.
At stage $t$, by Lemmas~\ref{7}, \ref{8} and~\ref{13},
the number of suitable blocks $w^*$ is at least equal ~to
\[
\nu_t = \frac{1}{4} \, (s_{j_t}^* - 2)^{\base{a+\ell(j_t)}{s^*_{j_t}}-\base{a}{s^*_{j_t}}},
\]
since $\card U \ge s_{j_t}^* - 2$. 
Furthermore, the length of each interval is
\[
\delta_t = (s^*_{j_t})^{-\base{b_t}{s^*_{j_t}}}.
\]
Observe that
\[
\nu_t \ge \frac{1}{4} (s_{j_t}^*)^{-1} \, {\rm e}^{\eta \ell(j_t)}
\ge (s_{j_t}^*)^{-2} \, {\rm e}^{\eta \ell(j_t)}
\]
and
\[
(s_{j_t}^*)^{-1} \, {\rm e}^{-b_t} \le \delta_t \le {\rm e}^{-b_t}.
\]
If $j_{t+1} = j_t + 1$, we have in particular that
\[
\base{b_{t+1}}{r_0}-\base{b_t}{r_0}< \frac{\base{b_t}{r_0}}{j_{t+1}}.
\]
This gives
\[
{\rm e}^{b_{t+1} - b_t} < r_0^2 \,  {\rm e}^{b_t / j_{t+1}}
\]
and
\begin{equation}
\frac{\delta_t}{\delta_{t+1}} \le  s_{j_{t+1}}^* r_0^2 \,  {\rm e}^{b_t / j_{t+1}}.   \label{4.1} 
\end{equation}
If $j_{t+1} = j_t$, then we know that
\[
\base{b_{u+1}  + \ell (j_t) }{r_0}-\base{b_{u+1}}{r_0}< \frac{\base{b_{u+1}}{r_0}}{j_{t}},
\]
where $u$ is the largest integer such that $j_u = j_t - 1$. 
As seen above, this implies that
\[
{\rm e}^{\ell} < r_0^2 \,  {\rm e}^{b_{u+1} / j_{t+1}},
\]
so, for $b \ge b_{u+1}$,
\[
{\rm e}^{b + \ell - b} < r_0^2 \,  {\rm e}^{b / j_{t+1}},
\]
hence we get the same upper bound on $\delta_t / \delta_{t+1}$
as in (4.1).  

Furthermore, it is clear from the construction that
\[
\ell (j_0) + \ell (j_1) + \ldots + \ell (j_t) \le b_t \le
\log \Bigl( \prod_{h=0}^{j_t} \, (s_h^*)^5 \Bigr) + \ell (j_0) + \ell (j_1) + \ldots + \ell (j_t). 
\]
Also, putting $N_t = \nu_1 \ldots \nu_t$, we have
\[
N_t \ge \Bigl( \prod_{h=0}^{j_t} \, (s_h^*)^{-7} \Bigr) {\rm e}^{\eta b_t}.
\]
Since  the construction also ensures
\[
\base{1+\ell(j_t)}{r_0}-\base{1}{r_0} \ge 2\ceil{\log s^*_{j_t-1}+3\log s^*_{j_t}} (j_t+2),
\]
we get 
\[
s_{j_t}^* \le {\rm e}^{\ell(j_t) / 2j_t} \le {\rm e}^{b_t / j_t}. 
\]

Consequently, for any positive real number $v$
and any integer $t \ge 2$, we obtain
\[
\frac{\delta_{t-1}}{\delta_t} \, (N_t (\delta_t)^v)^{-1}
\le (s_{j_{t+1}}^*)^8 \, {\rm e}^{b_t / j_{t+1}} \, {\rm e}^{(v - \eta) b_t}
\le  {\rm e}^{9 b_t / j_{t+1}} \, {\rm e}^{(v - \eta) b_t}. 
\]
Since $j_t$ tends to infinity as $t$
increases and $(b_t)_{t \ge 0}$ is a strictly increasing sequence
of integers, the corresponding series converges for every $v < \eta$. 
It then follows from Lemma \ref{21} 
that the dimension of the set into consideration is not less than $\eta$.
Recalling that $\eta$ can be taken arbitrarily close to $1$, this proves the
last assertion of Theorem~\ref{1}. 

\bibliography{simply}
\end{document}